\documentclass{amsart}
\usepackage{amsfonts,amssymb,amsmath,amsthm}
\usepackage{url}
\usepackage{enumerate}

\newtheorem{thm}{{\bf Theorem}}[section]
\newtheorem{cor}[thm]{{\bf Corollary}}
\newtheorem{lem}[thm]{{\bf Lemma}}
\newtheorem{prop}[thm]{{\bf Proposition}}
\theoremstyle{definition}

\theoremstyle{remark}

\numberwithin{equation}{section}

\begin{document}
\title[Somewhere dense orbit of  abelian subgroup of diffeomorphisms maps acting on $\mathbb{C}^{n}$]{Somewhere dense orbit of  abelian subgroup of diffeomorphisms maps acting on $\mathbb{C}^{n}$}

\author{Yahya N'dao and Adlene Ayadi}

 \address{Yahya N'dao, University of Moncton, Department of mathematics and statistics, Canada}
 \email{yahiandao@yahoo.fr }

\address{Adlene Ayadi, University of Gafsa, Faculty of sciences, Department of Mathematics,Gafsa, Tunisia.}
\email{adlenesoo@yahoo.com; \ \ \ \ \ \ Web page:
www.linearaction.blogspot.com}

\thanks{This work is supported by the research unit: syst\`emes dynamiques et combinatoire: 99UR15-15}

\subjclass[2000]{37C85, 47A16, 17B45}

\keywords{ diffeomorphisms commute, action group, abelian group,
somewhere dense, locally dense orbit\dots}

\maketitle

\begin{abstract}

In this paper, we give a characterization for any abelian subgroup
$G$ of a lie group of diffeomorphisms maps of $\mathbb{C}^{n}$,
having a somewhere dense orbit $G(x)$, $x\in \mathbb{C}^{n}$:
$G(x)$ is somewhere dense in $\mathbb{C}^{n}$ if and only if there
are $f_{1},\dots,f_{2n+1}\in exp^{-1}(G)$ such that $f_{2n+1}\in
vect(f_{1},\dots,f_{2n})$ and
$\mathbb{Z}f_{1}(x)+\dots+\mathbb{Z}f_{2n+1}(x)$ is dense in
$\mathbb{C}^{n}$, where  $vect(f_{1},\dots,f_{2n})$ is the vector
space over $\mathbb{R}$ generated by $f_{1},\dots,f_{2n}$.
\end{abstract}

\section{{\bf Introduction} }

Denote by $Diff^{r}(\mathbb{C}^{n})$, $r\geq 1$ the group of all
$C^{r}$-diffemorphisms of $\mathbb{C}^{n}$. Let $\Gamma$ be a lie
subgroup of $Diff^{r}(\mathbb{C}^{n})$, $r\geq 1$ and $G$ be an
abelian subgroup of $\Gamma$, such that $Fix(G)\neq \emptyset$,
where $Fix(G)=\{x\in \mathbb{C}^{n}:\ f(x)=x,\ \forall f\in G\}$
be the global fixed point set of $G$. There is a natural action
$G\times \mathbb{C}^{n} \longrightarrow \mathbb{C}^{n}$. $(f, x)
\longmapsto f(x)$. For a point $x\in\mathbb{C}^{n}$, denote by
$G(x) =\{f(x), \ f\in G\}\subset \mathbb{C}^{n}$ the orbit of $G$
through $x$. A subset $E\subset\mathbb{C}^{n}$  is called
$G$-invariant if $f(E)\subset E$ for any $f\in G$; that is $E$ is
a union of orbits. Denote by $\overline{E}$ (resp.
$\overset{\circ}{E}$ ) the closure (resp. interior) of
$E$.\medskip

 Recall that
$E\subset \mathbb{C}^{n}$ is somewhere dense in $\mathbb{C}^{n}$
if the closure $\overline{E}$
 has nonempty interior in $\mathbb{C}^{n}$.   An orbit $\gamma$ is called somewhere dense  (or locally
dense) if \  $\overset{\circ}{\overline{\gamma}}\neq\emptyset$.
The group $G$ is called hypercyclic if it has a dense orbit in
$\mathbb{C}^{n}$. Hypercyclic is also called topologically
transitive.

 \medskip

 The purpose of this paper is to give a characterization for any subgroup $G$ of
a lie group of diffeomorphisms maps of $\mathbb{C}^{n}$, having a
dense orbit.  In \cite{aAhM05}, the authors present a global
dynamic of every abelian subgroup of $GL(n, \mathbb{C})$  and in
\cite{aAh-M05}, they  characterize hypercyclic abelian subgroup of
$GL(n, \mathbb{C})$. Our main result is viewed as a continuation
 of \cite{YNAA-C} and \cite{YNAA-D}.\
\\
Denote by:
\\
- $\mathbb{C}^{*}=\mathbb{C}\backslash\{0\}$ and
$\mathbb{N}^{*}=\mathbb{N}\backslash\{0\}$.\
\
\\
- $C^{r}(\mathbb{C}^{n}, \mathbb{C}^{n})$ the set of all
$C^{r}$-differentiable maps of $\mathbb{C}^{n}$. \ \\  - For a
subset $E\subset \mathbb{C}^{n}$ (resp. $E\subset
C^{r}(\mathbb{C}^{n}, \mathbb{C}^{n})$), denote by $vect(E)$ the
vector subspace of $\mathbb{C}^{n}$ (resp. $C^{r}(\mathbb{C}^{n},
\mathbb{C}^{n})$) over $\mathbb{R}$ generated by all elements of
$E$. \ \\ - $exp:\ C^{r}(\mathbb{C}^{n},
\mathbb{C}^{n})\longrightarrow Diff^{r}(\mathbb{C}^{n})$ the
exponential map defined by $exp(f)=e^{f}$, $f\in
C^{r}(\mathbb{C}^{n}, \mathbb{C}^{n})$.\ \\
\\
- $H$  the lie algebra associated to $\Gamma$. \ \\ -
$exp:H\longrightarrow \Gamma$ be the exponential map.\
\\
- $H_{x}=\{f(x),\ \ B\in H\}$,  it is a vector subspace of
$\mathbb{C}^{n}$ over $\mathbb{R}$.\ \\ -
$\mathrm{g}=exp^{-1}(G)$, it is an additive group  because $G$ is
abelian. \
\\
- $\mathrm{g}_{x}=\{f(x),\ \ B\in\mathrm{g}\}$,  it is an additive
subgroup of $\mathbb{C}^{n}$ because $\mathrm{g}$ is an additive
group.

\medskip

Our principal results can  be stated as follows:

\begin{thm}\label{T:1} Let  $\Gamma$  be an abelian lie  subgroup of $Diff^{r}(\mathbb{C}^{n})$ and $x\in \mathbb{C}^{n}\backslash\{0\}$.
Then the following assertions are equivalent:\ \\ (i)\
$H_{x}=\mathbb{C}^{n}$. \ \\ (ii)
$\overset{\circ}{\overline{\Gamma(x)}}\neq\emptyset$.
\end{thm}
\medskip

In general, the Lie algebra $\widetilde{\mathrm{g}}$ is not
explicitly defined, so we give an explicitly test to the existence
of somewhere dense orbit by the following theorem:

\begin{thm}\label{T:3} Let  $G$  be an abelian subgroup of a lie  group $\Gamma\subset Diff^{r}(\mathbb{C}^{n})$  and $x\in \mathbb{C}^{n}\backslash\{0\}$.
Then $\overset{\circ}{\overline{G(x)}}\neq\emptyset$ if and only
if there exist $f_{1},\dots, f_{2n+1}\in exp^{-1}(\widetilde{G})$
such that $f_{2n+1}\in vect(f_{1},\dots, f_{2n})$ and
$\mathbb{Z}f_{1}(x)+ \dots+ \mathbb{Z}f_{2n+1}(x)$ is a dense
additive subgroup of $\mathbb{C}^{n}$.
\end{thm}
\medskip

Let's introduce the arithmetic property: We say that
$f_{1},\dots,f_{2n+1}\in C^{r}(\mathbb{C}^{n}, \mathbb{C}^{n})$
satisfy \emph{property $\mathcal{D}(x)$} for some $x\in
\mathbb{C}^{n}$ if $f_{1},\dots, f_{2n}$ are linearly independent,
$f_{2n+1}\in vect(f_{1},\dots, f_{2n})$ and for every
$(s_{1},\dots,s_{2n+1})\in \mathbb{Z}^{2n+1}\backslash\{0\}$: \
$$\mathrm{rank}\left[\begin{array}{cccccc}
             \mathrm{Re}(f_{1}(x)) & \dots & \mathrm{Re}(f_{2n+1}(x)) \\
             \mathrm{Im}(f_{1}(x)) & \dots & \mathrm{Im}(f_{2n+1}(x)) \\
             s_{1} & \dots& s_{2n+1}
           \end{array}
\right] = 2n+1.$$
\medskip

For a vector $v\in\mathbb{C}^{n}$, we write $v = \mathrm{Re}(v)+
i\mathrm{Im}(v)$ where $\mathrm{Re}(v)$ and $\mathrm{Im}(v)\in
\mathbb{R}^{n}$.
\medskip

As an immediate consequence of Theorem~\ref{T:3}, we have:
\begin{cor}\label{C:5} Let  $G$  be an abelian subgroup of a lie  group $\Gamma\subset Diff^{r}(\mathbb{C}^{n})$  and $x\in \mathbb{C}^{n}\backslash\{0\}$. Then $\overset{\circ}{\overline{G(x)}}\neq\emptyset$
if and only if there exist $f_{1},\dots, f_{2n+1}\in exp^{-1}(G)$
and satisfying property $\mathcal{D}(x)$.
\end{cor}
\medskip

As an important consequence of the Theorem~\ref{T:3}, we give the
following Corollary which simplifies the test given by Theorem 1.3
proved in \cite{aAh-M05} for the abelian subgroup of
$GL(n,\mathbb{C})$:
\medskip

\begin{cor}\label{C:6} Let  $G$  be an abelian subgroup of  $GL(n,\mathbb{C})$  and $x\in
\mathbb{C}^{n}\backslash\{0\}$. Then
$\overline{G(x)}=\mathbb{C}^{n}$ if and only if there exist
$B_{1},\dots, B_{2n+1}\in exp^{-1}(G)$
  such that $\mathbb{Z}B_{1}x+\dots +\mathbb{Z}B_{2n+1}x$  is dense in
  $\mathbb{C}^{n}$.
\end{cor}
\medskip

This paper is organized as follows: In Section 2 we prove
Theorem~\ref{T:1}. Section3 is devoted to prove Theorem ~\ref{T:3}
and Corollaries \ref{C:5}, \ref{C:6}.
\bigskip

\section{\bf Proof of Theorem ~\ref{T:1} } We will cite the
definition of the exponential map given in ~\cite{IL}.

\subsection{\bf Exponential map } In this section, we illustrate
the theory developed of the group $Diff(\mathbb{C}^{n})$ of
diffeomorphisms of $\mathbb{C}^{n}$. For simplicity, throughout
this section we only consider the case of $\mathbb{C}=\mathbb{R}$;
however, all results also hold for complexes case. The group
$Diff(\mathbb{R}^{n})$ is not a Lie group (it is
infinite-dimensional), but in many way it is similar to Lie
groups. For example, it easy to define what a smooth map from some
Lie group $G$ to $Diff(\mathbb{R}^{n})$ is: it is the same as an
action of $G$ on $\mathbb{R}^{n}$ by diffeomorphisms. Ignoring the
technical problem with infinite-dimensionality for now, let us try
to see what is the natural analog of the Lie algebra $\mathrm{g}$
for the group $G$. It should be the tangent space at the identity;
thus, its elements are derivatives of one-parameter families of
diffeomorphisms.\ \\ Let $\varphi^{t}: G\longrightarrow G$ be
one-parameter family of diffeomorphisms. Then, for every point
$a\in G$, $\varphi^{t}(a)$ is a curve in $G$ and thus
$\frac{\partial}{\partial t}\varphi^{t}(a)_{/t=0}=\xi(a)\in
T_{a}G$ is a tangent vector to $G$ at $m$. In other words,
$\frac{\partial}{\partial t}\varphi^{t}$ is a vector field on
$G$.\medskip

\ \\ The exponential map $exp: \mathrm{g}\longrightarrow G$ is
defined by $exp(x)=\gamma_{x}(1)$ where $\gamma_{x}(t)$ is the
one-parameter subgroup with tangent vector at $1$ equal to $x$.\
\\ If $\xi\in\mathrm{g}$ is a vectorfield, then $exp(t\xi)$ should
be one-parameter family of diffeomorphisms whose derivative is
vector field $\xi$. So this is the solution of differential
equation $$\frac{\partial}{\partial
t}\varphi^{t}(a)_{/t=0}=\xi(a).$$ \ \\ In other words,
$\varphi^{t}$ is the time $t$ flow of the vector field. Thus, it
is natural to define the Lie algebra of $G$ to be the space
$\mathrm{g}$ of all smooth vector $\xi$ fields on $\mathbb{R}^{n}$
such that $exp(t\xi)\in G$ for every $t\in \mathbb{R}$.

\bigskip

We will use the definition of Whitney topology given in
~\cite{WDM}.

\subsection{{\bf Whitney Topology on $C^{0}(\mathbb{C}^{n},\mathbb{C}^{n})$}}
 For each open subset $U\subset \mathbb{C}^{n}\times \mathbb{C}^{n}$ let $\widetilde{U}\subset \mathcal{C}^{0}(\mathbb{C}^{n}, \mathbb{C}^{n})$ be the
set of continuous functions $g$, whose graphs
$\{(x,g(x))\in\mathbb{C}^{n}\times \mathbb{C}^{n},\ \ x\in
\mathbb{C}^{n}\}$ is contained in $U$. \medskip We want to
construct a neighborhood basis of each function
$f\in\mathcal{C}^{0}(\mathbb{C}^{n}, \mathbb{C}^{n})$. Let
$K_{j}=\{x\in \mathbb{C}^{n},\ \ \|x\|\leq j\}$ be a countable
family of compact sets (closed balls with center 0) covering
$\mathbb{C}^{n}$ such that $K_{j}$ is contained in the interior of
$K_{j+1}$. Consider then the compact subsets $L_{j}
=K_{j}\backslash \overset{\circ}{\overbrace{K_{j-1}}}$, which are
compact sets, too. Let $\epsilon= (\varepsilon_{j})_{j}$ be a
sequence of positive numbers and then define $$V_{(f;\epsilon)}
=\{f\in\mathcal{C}^{0}(\mathbb{C}^{n}, \mathbb{C}^{n})\ :\
\|f(x)-g(x)\|< \varepsilon_{j},\ \mathrm{for\ any} \ x\in L_{j},\
\forall j\}.$$ We claim this is a neighborhood system of the
function $f$ in $\mathcal{C}^{0}(\mathbb{C}^{n}, \mathbb{C}^{n})$.
Since $L_{i}$ is compact, the set
$U=\{(x,y)\in\mathbb{C}^{n}\times\mathbb{C}^{n}\ :\ \|f(x)-g(x)\|<
\varepsilon_{j},\ if\  x\in L_{j}\}$ is open. Thus,
$V_{(f;\epsilon)}= \widetilde{U}$ is an open neighborhood of $f$.
On the other hand, if $O$ is an open subset of
$\mathbb{C}^{n}\times\mathbb{C}^{n}$ which contains the graph of
$f$, then since $L_{j}$ is compact, it follows that there exists
$\varepsilon_{j}>0$ such that if $x\in L_{j}$ and $\|y-f(x)\|<
\varepsilon_{j}$, then $(x;y)\in O$. Thus, taking
$\widetilde{\epsilon}= (\varepsilon_{j})_{j}$ we have
$V_{(f;\widetilde{\epsilon})}\subset \widetilde{O}$, so we have
obtained the family $V_{(f;\epsilon)}$ is a neighborhood system of
$f$. Moreover, for each given $\epsilon= (\varepsilon_{j})_{j}$,
we can find a $C ^{\infty}$-function $\epsilon:
\mathbb{C}^{n}\longrightarrow\mathbb{R}_{+}$, such that
$\epsilon(x)< \varepsilon_{j}$ for any $x\in L_{j}$. It follows
that the family $V_{(f;\epsilon)}
=\{(x,y)\in\mathbb{C}^{n}\times\mathbb{C}^{n}\ :\ \|f(x)-g(x)\|<
\epsilon(x)\}$ is also a neighborhood system.
\medskip
\
\\
Denote by:\ \\ - $\widetilde{G}=\overline{G}\cap
Diff^{r}(\mathbb{C}^{n})$, where $\overline{G}$ is the closure of
$G$ in $C^{r}(\mathbb{C}^{n}, \mathbb{C}^{n})$ for the withney
topology defined above. So $\widetilde{G}$ is an abelian lie
subgroup of $\Gamma$.\ \\ - $\mathcal{A}(\widetilde{G})$ the
algebra generated by $G$. See that $G\subset
\mathcal{A}(\widetilde{G})$.\
\\ - $\Phi_{x}:
\mathcal{A}(\widetilde{G})\longrightarrow \mathbb{C}^{n}$ the
linear map given by $\Phi_{x}(f)=f(x)$, $f\in
\mathcal{A}(\widetilde{G})$.\ \\ -
$E(x)=\Phi_{x}(\mathcal{A}(G))$.\medskip

\begin{lem}\label{L:aaaaa101001} The linear map
$\Phi_{x}:\mathcal{A}(\widetilde{G})\longrightarrow E(x)$ is
continuous.
\end{lem}
\medskip

\begin{proof} Firstly, we take the restriction of the Whitney
topology to $\mathcal{A}(\widetilde{G})$. Secondly, let $f\in
\mathcal{A}(\widetilde{G})$ and $\varepsilon>0$. Then for
$\epsilon=(\varepsilon_{j})_{j}$ with
$\varepsilon_{j}=\varepsilon$ and for $V_{(f;\epsilon)}$ be a
neighborhood system of $f$, we obtain: for every $g\in
V_{(f;\epsilon)}\cap \mathcal{A}(\widetilde{G})$ and for every
$y\in L_{j}$, $\|f(y)-g(y)\|<\varepsilon$, $\forall j$. In
particular for $j=j_{0}$ in which $x\in L_{j_{0}}$, we have
$\|f(x)-g(x)\|<\varepsilon$, so
$\|\Phi_{x}(f)-\Phi_{x}(g)\|<\varepsilon$. It follows that
$\Phi_{x}$ is continuous.
\end{proof}
\medskip

\subsection{\bf Proof of Theorem ~\ref{T:1} }

\medskip

\begin{prop}\label{p:1} $($\cite{IL}, Theorem 3.29$)$ Let $G$ be a Lie group acting on  $\mathbb{C}^{n}$ with
 lie algebra $\widetilde{\mathrm{g}}$  and let $u \in \mathbb{C}^{n}$.\
 \\
(i) The stabilizer $G_{x} = \{B \in G :\ \ Bu = u\}$ is a closed
Lie subgroup in $G$, with Lie algebra $\mathfrak{h}_{x} = \{B \in
\widetilde{\mathrm{g}}:\  Bu = 0\}$. \
\\
(ii) The map $G_{/G_{x}}\longrightarrow \mathbb{C}^{n}$ given by
$B.G_{x}\longmapsto Bu$ is an immersion. Thus,
 the orbit $G(x)$ is an immersed submanifold in $\mathbb{C}^{n}$.
 In particular
 $\mathrm{dim}(G(x))=\mathrm{dim}(\widetilde{\mathrm{g}})-\mathrm{dim}(\mathfrak{h}_{x})$.
\end{prop}\
\medskip
\ \\ Here $\mathfrak{h}_{x}=Ker(\Phi_{x})$ since
$Ker(\Phi_{x})\subset \widetilde{\mathrm{g}}$. Write: \
\\
- $\widetilde{L}$ the vector subspace of $\widetilde{\mathrm{g}}$
supplement to $Ker(\Phi_{x})$, (i.e.
 $\widetilde{L}\oplus
Ker(\Phi_{x})=\widetilde{\mathrm{g}}$).\ It is clear that
dim$(\widetilde{L})=cod(Ker(\Phi_{x}))\leq n$, then
$\widetilde{L}$ is closed.\
\\
- $\ exp:\widetilde{L}\oplus Ker(\Phi_{x})\longrightarrow
\widetilde{G}$ the exponential map. Since $\widetilde{G}$ is
abelian, so is $\widetilde{\mathrm{g}}$, then $\
exp(f+h)=exp(f)\circ exp(h)$ for every $f\in \widetilde{L}$ and
$h\in Ker(\Phi_{x})$.\
\\
- $\widetilde{G}_{x}$ the stabilizer of $\widetilde{G}$ on the
point $u$. So it is a lie subgroup of $\widetilde{G}$ with lie
algebra $Ker(\Phi_{x})$.
 \bigskip

As a directly consequence of Proposition 5.13, given in
\cite{ASRW}, applied to $\Gamma$, we have the following Lemma:

\begin{lem}\label{L:001000+}$($\cite{ASRW}, Proposition 5.13$)$ Let $G$ be an abelian subgroup of a lie group $\Gamma$. There exists an open neighborhood $U$ of $0$ in
$H$ such that $\ exp: U\longrightarrow \ exp(U)$ is a
diffeomorphism and $\ exp(U\cap
  \widetilde{\mathrm{g}})=\ exp(U)\cap \widetilde{G}$.
 \end{lem}
\
\\
Denote by  $V=exp(U)$, where $U$ is the open set defined in
Lemma~\ref{L:001000+}.\medskip

\begin{lem}\label{L:000001010} We have  $\overline{G(x)}=\overline{\widetilde{G}(x)}$.
\end{lem}
\medskip

\begin{proof}   It is clear that $\overline{G(x)}\subset\overline{\widetilde{G}(x)}\subset\overline{\overline{G}(x)}$.
Let $v\in \overline{\overline{G}(x)}$, so
$v=\underset{m\to+\infty}{lim}f_{m}(x)$ for some sequence
$(f_{m})_{m\in \mathbb{N}}$ in $\overline{G}$. Then for every
$m\in\mathbb{N}$, there exists a sequence
 $(f_{m,k})_{k\in\mathbb{N}}$ in $G$ such that $\underset{k\to
 +\infty}{lim}f_{m,k}=f_{m}$, so by continuity of $\Phi_{x}$ (Lemma~\ref{L:aaaaa101001}), we
 have $\underset{k\to
 +\infty}{lim}f_{m,k}(x)=f_{m}(x)$,
thus for every $\varepsilon>0$, \ there exists $M>0$ and  for
every $m\geq M$, there exists $k_{m}>0$,
 such that $\|f_{m}(x)-v\|<\frac{\varepsilon}{2}$ and  \
for every $k\geq k_{m},\ \ \
\|f_{m,k}(x)-f_{m}(x)\|<\frac{\varepsilon}{2}$.\ Then,  for every
$m>M$, $$\|f_{m,k_{m}}(x)-v\|\leq \|f_{m,k_{m}}(x)-f_{m}(x)\|+
\|f_{m}(x)-v\|<\varepsilon,$$ therefore $\underset{m\to
+\infty}{lim}f_{m,k_{m}}(x)=v$. Hence $v\in \overline{G(x)}$. It
follows that
 $\overline{\widetilde{G}(x)}\subset\overline{\overline{G}(x)}\subset
 \overline{G(x)}$.
\end{proof}
\medskip

\begin{lem}\label{L:2} Let $W=\Phi_{x}(V)$. Then  $\Phi_{x}^{-1}(\widetilde{G}(x)\cap W)\cap V=\widetilde{G}\cap V$.
\end{lem}
\medskip

\begin{proof}  Since $W=\Phi_{x}(V)$, it is obvious that $\widetilde{G}\cap
V\subset\Phi_{x}^{-1}(\widetilde{G}(x)\cap W)\cap V$. Let $f\in
\Phi_{x}^{-1}(\widetilde{G}(x)\cap W)$. Then there exists $g\in
\widetilde{G}\cap V$ such that $f(x)=g(x)$. So $g^{-1}\circ
f(x)=x$. Hence $g^{-1}\circ f \in H_{x}$, where $H_{x}$ be the lie
group generated by $\{h\in Diff^{r}(\mathbb{C}^{n}):\ h(x)=x\}\cap
\mathcal{A}(\widetilde{G})$. So $H_{x}$ is contained in the
stabilizer of $Diff^{r}(\mathbb{C}^{n})$ on $x$. Set $L_{x}$ be
the lie algebra of $H_{x}$, so $L_{x}\subset\{h\in
Diff^{r}(\mathbb{C}^{n}):\ h(x)=0\}\cap
\mathcal{A}(\widetilde{G})$. Therefore $L_{x}\subset
Ker(\Phi_{x})\subset \widetilde{\mathrm{g}}$. Hence $H_{x}\subset
\widetilde{G}$. It follows that $g^{-1}\circ f \in \widetilde{G}$,
so $f\in \widetilde{G}\cap V$. This completes he proof.
\end{proof}\
\\
\\
{\it Proof of Theorem~\ref{T:1}.}\
\\
Since $\widetilde{G}$ is a locally closed sub-manifold of
$Diff^{r}(\mathbb{C}^{n})$. By Proposition ~\ref{p:1}.(ii),
$\widetilde{G}(x)$ is an immersed submanifold of $\mathbb{C}^{n}$
with dimension
$r=\mathrm{dim}(\widetilde{\mathrm{g}})-\mathrm{dim}(Ker(\Phi_{x}))$.
  We have $\mathrm{Im}(\Phi_{x})=\widetilde{\mathrm{g}}_{x}$. Then
$\mathrm{dim}(\widetilde{\mathrm{g}}_{x})=\mathrm{dim}(\widetilde{\mathrm{g}})-\mathrm{dim}(Ker(\Phi_{x})).$
 It follows from Proposition~\ref{p:1},(ii) that
 $$\mathrm{dim}(\widetilde{G}(x))=\mathrm{dim}(\widetilde{\mathrm{g}}_{x}) \ \ \ (2)$$.
 \\ {\it Proof of $(i) \Longrightarrow (iii):$} The proof results directly from
 (2), and the fact that $dim(\widetilde{G}(x))=n$ if and only if $\widetilde{G}(x)$ is a non empty open set. \\
  {\it Proof of $(iii)\Longrightarrow (ii):$} Since $\widetilde{G}(x)\cap W$ is a non empty open set then the proof follows
  directly from  Lemma ~\ref{L:000001010}.(ii), because
   $\widetilde{G}(x)\cap W\subset\overline{\widetilde{G}(x)}\cap W=\overline{G(x)}\cap W$.\ \\
  {\it Proof of $(ii)\Longrightarrow (i):$} Since
  $\overset{\circ}{\overline{G(x)}}\subset Im(\Phi_{x})\subset \mathbb{C}^{n}$
  then the linear map
  $\Phi_{x}: \mathcal{A}(\widetilde{G})\longrightarrow\mathbb{C}^{n}$ is surjective, so it is an open map. By Lemma
\ref{L:001000+} there exists two open subsets $U$ and $V=exp(U)$
respectively of $H$ and $\Gamma$ such that the exponential map $\
exp:U
   \longrightarrow V $ is a diffeomorphism and satisfying $exp(\widetilde{\mathrm{g}}\cap U)=\widetilde{G}\cap
   V$. So

  $$\ exp^{-1}(\widetilde{G}\cap V)=\widetilde{\mathrm{g}}\cap U.\ \ \ \ \ \  (1)$$  Recall that $W=\Phi_{x}(V)$. Since $\Phi_{x}$ is an open
  map and by Lemma~\ref{L:000001010}.(i), $\overset{\circ}{\overline{G(x)}}=\overset{\circ}{\overline{\widetilde{G}(x)}}$, so
  \begin{align*}
  \Phi_{x}^{-1}(\overset{\circ}{\overline{G(x)}}\cap
  W)& =\Phi_{x}^{-1}(\overset{\circ}{\overline{\widetilde{G}(x)}}\cap W)\\
  \ & \subset \Phi_{x}^{-1}(\overline{\widetilde{G}(x)\cap W})\\
   \ & \subset \overline{\Phi_{x}^{-1}(\widetilde{G}(x)\cap W)} \
   \ \ \ \ (3)
    \end{align*}\
We have
     \begin{align*}
    \Phi_{x}\circ \ exp^{-1}(\Phi_{x}^{-1}(\overset{\circ}{\overline{G(x)}}\cap W)\cap V) &  \subset \Phi_{x}\circ \ exp^{-1}(\overline{\Phi_{x}^{-1}(\widetilde{G}(x)\cap W)\cap V})\ \ \ \  \ (\mathrm{by} \ (3))\\
   \ & \subset \Phi_{x}\circ \ exp^{-1}(\overline{\widetilde{G}\cap V}),\ \ \ \ \ \ \ \ \ \ \  \ \ \ \ \  \ \ \ \ \ ( \mathrm{by \ Lemma} ~\ref{L:2}) \\
    \ & \subset \Phi_{x}\circ \overline{\ exp^{-1}(\widetilde{G}\cap
    V)}\\
   \ & \subset  \Phi_{x}(\overline{ \widetilde{\mathrm{g}}\cap U
 })\ \ \ \ \ \ \ \ \ \ \  \ \ \ \ \ \ \ \ \ \ \ \ \ \ \ \ \ \ \ \ \ \  \ \ \ (\mathrm{by}\ (1))\\
  \ &  \subset  \widetilde{\mathrm{g}}_{x}
  \end{align*}
  Since $\overset{\circ}{\overline{G(x)}}\cap W$ is a non empty
  open subset of $\mathbb{C}^{n}$ then $\Phi_{x}\circ exp^{-1}(\Phi_{x}^{-1}(\overset{\circ}{\overline{G(x)}}\cap W)\cap V)$
    is an open subset of $\mathbb{C}^{n}$. It follows that $\widetilde{\mathrm{g}}_{x}=\mathbb{C}^{n}$. The proof is completed
    $\hfill{\Box}$.

\section{{\bf Proof of Theorem ~\ref{T:3} and
Corollaries \ref{C:5}, \ref{C:6}}}
\medskip

Under the notation of Lemma~\ref{L:001000+}, recall that there
exists an open subset $U$ of $\mathcal{A}(\widetilde{G})$ such
that $\ exp: U\longrightarrow \ exp(U)$ is a diffeomorphism. Now,
by using the restricton of the withney topology to
$\mathcal{A}(\widetilde{G})$, denote by: \ \\ - $B_{(0, r)}=\{f\in
\mathcal{A}(\widetilde{G}):\ \ \|f\|< r\}$, the open ball with
center $0$ and radius $r>0$ . \ \\ - $r_{G}=sup\{r\in ]0,1[:\ \
B_{(0,r)}\subset U\}$, it is dependent of $G$ since is $U$.
\medskip

\begin{thm}\label{T:2} Let  $G$  be a  subgroup of $Diff^{r}(\mathbb{C}^{n})$ and $x\in \mathbb{C}^{n}$.
If there exist $f_{1},\dots, f_{2n}\in exp^{-1}(\widetilde{G})$
with $\|f_{k}\|<r_{\widetilde{G}}$, for every $k=1,\dots, 2n$ such
that $(f_{1}(x), \dots,f_{2n}(x))$ is a basis of $\mathbb{C}^{n}$
over $\mathbb{R}$, then
$\overset{\circ}{\overline{G(x)}}\neq\emptyset$.
\end{thm}
\medskip

\begin{proof} We have   $f_{k}\in exp^{-1}(\widetilde{G})$ with
$\|f_{k}\|< r_{\widetilde{G}}$ for every $k=1,\dots, 2n$, then
$f_{1},\dots, f_{2n}\in U$ and so $e^{f_{k}}\in \widetilde{G}\cap
V$. By Lemma ~\ref{L:001000+},
 $\widetilde{G}\cap V=exp(U\cap
  \widetilde{\mathrm{g}})$, hence $f_{k}\in\widetilde{\mathrm{g}}$, for every $k=1,\dots, 2n$. As $(f_{1}(x),\dots, f_{2n}(x))$ is a basis of
  $\mathbb{C}^{n}$ over $\mathbb{R}$ then
$\widetilde{\mathrm{g}_{x}}=\mathbb{C}^{n}$.
 It follows by Theorem ~\ref{T:1} that $\overset{\circ}{\overline{G(x)}}\neq\emptyset$.\
\end{proof}

\medskip

\begin{lem}\label{L:BBB017} Let $H$ be a vector space with dimension $2n$ over $\mathbb{R}$ and $x_{1},\dots, x_{2n+1}\in H$, such that
 $\mathbb{Z}x_{1}+\dots+\mathbb{Z}x_{2n+1}$ is dense in $H$. Then for every $1\leq k\leq 2n+1$,
  $(x_{1},\dots, x_{k-1},x_{k+1},\dots, x_{2n+1})$ is a basis of $H$ over $\mathbb{R}$.
\end{lem}
\medskip

\begin{proof} We have $H$ is isomorphic to $\mathbb{C}^{n}$. Let $1\leq k \leq 2n+1$ be a fixed integer and
take  $$K=vect(x_{1},\dots, x_{k-1},x_{k+1},\dots, x_{2n+1}).$$
Suppose that $\mathrm{dim}(K)=p<2n$.  Let $(x_{k_{1}},\dots,
x_{k_{p}})$ be a basis of $K$.
  Then $\mathbb{Z}x_{1}+\dots+\mathbb{Z}x_{2n+1}\subset K+\mathbb{Z}x_{k}$ which cannot be dense in $H$, a contradiction.
\end{proof}
\medskip

Recall that $\widetilde{L}$ is the vector subspace of
$\widetilde{\mathrm{g}}$ supplement to $Ker(\Phi_{x})$, (i.e.
 $\widetilde{L}\oplus
Ker(\Phi_{x})=\widetilde{\mathrm{g}}$).\ Denote by: \ \\ - $p_{x}:
\widetilde{L}\oplus Ker(\Phi_{x})\longrightarrow \widetilde{L}$
given by $p_{x}(f+h)=f$, $f\in \widetilde{L}$ and $h\in
Ker(\Phi_{x})$.\ \\
\medskip

\begin{lem}\label{L:9}  Under above notations, we have:
\begin{itemize}
  \item [(i)] The linear map $\Phi_{x}:\widetilde{L}\longrightarrow E(x)$
 defined by $\Phi_{x}(f)=f(x)$, is an isomorphism.
    \item [(ii)] for every $f\in \widetilde{\mathrm{g}}$ one has  $\Phi_{x}^{-1}(f(x))=p_{x}(f)$.
\end{itemize}
\end{lem}
\medskip

\begin{proof} (i) By construction $\Phi_{x}$ is surjective and restreint to $\widetilde{L}$ it became injective. By Lemma~\ref{L:aaaaa101001} $\Phi_{x}$ is continuous and bijective. Hence it is an isomorphism
because it is linear.\ \\    (ii) Let $f\in
\widetilde{\mathrm{g}}$. Write $f=f_{1}+f_{0}$ with
$f_{1}=p_{x}(f)\in  \widetilde{L}$
  and $f_{0}\in Ker(\Phi_{x}))$. Since $f_{0}(x)=0$, so $f(x)=f_{1}(x)$. By (i), $\Phi_{x}$ is an isomorphism, then
  $\Phi_{x}^{-1}(f(x))=\Phi_{x}^{-1}(f_{1}(x))=f_{1}=p_{x}(f)$.
  This completes the proof.
\end{proof}
\bigskip

 Let $f_{1},\dots, f_{2n+1}\in \widetilde{\mathrm{g}}$ and suppose that $(p_{x}(f_{1}),\dots, p_{x}(f_{2n}))$ is a basis of \ $\widetilde{L}$ over $\mathbb{R}$ and $f_{n+1}\in vect(f_{1},\dots, f_{2n})$.
  Denote by $\Psi: \widetilde{L}\longrightarrow
\widetilde{\mathrm{g}}$ the linear map given by
$$\Psi\left(\underset{k=1}{\overset{2n}{\sum}}\alpha_{k}p_{x}(f_{k})\right)=\underset{k=1}{\overset{2n}{\sum}}\alpha_{k}f_{k}.$$\
\medskip

\begin{lem}\label{L:018} Under above notations, we have:\
\\
(i) If $\overline{\mathbb{Z}f_{1}(x)+\dots
+\mathbb{Z}f_{2n+1}(x)}=\mathbb{C}^{n}$ then
 $\overline{\mathbb{Z}p_{x}(f_{1})+\dots +\mathbb{Z}p_{x}(f_{2n+1})}=\widetilde{L}$.
\
\\
(ii) $\Psi\left(\mathbb{Z}p_{x}(f_{1})+\dots+
\mathbb{Z}p_{x}(f_{2n+1})\right)=\mathbb{Z}f_{1}+\dots+
\mathbb{Z}f_{2n+1}$.
\end{lem}
\medskip

\begin{proof} (i) Here $E(x)=\mathbb{C}^{n}$. By Lemma ~\ref{L:9},(i), $\Phi_{x}:\widetilde{L}\longrightarrow \mathbb{C}^{n}$
is an isomorphism and by Lemma ~\ref{L:9},(ii), we have
$\Phi_{x}^{-1}(f_{k}(x))=p_{x}(f_{k})(x)$
 for every $k=1,\dots, 2n+1$, so $$\Phi_{x}^{-1}(\mathbb{Z}f_{1}(x)+\dots +\mathbb{Z}f_{2n+1}(x))= \mathbb{Z}p_{x}(f_{1})+\dots +\mathbb{Z}p_{x}(f_{2n+1}).$$ Then \begin{align*}
\widetilde{L} & =\Phi_{x}^{-1}(\mathbb{C}^{n})\\ \ &
=\Phi_{x}^{-1}(\overline{\mathbb{Z}f_{1}(x)+\dots
+\mathbb{Z}f_{2n+1}(x)})\\ \ & =
\overline{\mathbb{Z}\Phi_{x}^{-1}(f_{1}(x))+\dots
+\mathbb{Z}\Phi_{x}^{-1}(f_{2n+1}(x)}))\\ \ &
=\overline{\mathbb{Z}p_{x}(f_{1})+\dots
+\mathbb{Z}p_{x}(f_{2n+1})}
\end{align*}
\\
 (ii)  Let  $k_{1},\dots, k_{2n+1}\in \mathbb{Z}$ and $f=k_{1}p_{x}(f_{1})+\dots + k_{2n+1}p_{x}(f_{2n+1})$. Write $f_{2n+1}=\underset{k=1}{\overset{2n}{\sum}}\alpha_{k}f_{k}$,
 $\alpha_{1},\dots, \alpha_{2n}\in \mathbb{R}$, then $$f=(k_{1}+\alpha_{1}k_{2n+1})p_{x}(f_{1})+\dots + (k_{2n}+\alpha_{2n}k_{2n+1})p_{x}(f_{2n}),$$ so
  \begin{align*}
  \Psi(f) & =\Psi\left((k_{1}+\alpha_{1}k_{2n+1})p_{x}(f_{1})+\dots + (k_{2n}+\alpha_{2n}k_{2n+1})p_{x}(f_{2n})\right)\\
  \ & =(k_{1}+\alpha_{1}k_{2n+1})f_{1}+\dots + (k_{2n}+\alpha_{2n}k_{2n+1})f_{2n}\\
  \ & = k_{1}f_{1}+\dots + k_{2n+1}f_{2n+1}
  \end{align*}
  Then $\Psi\left(\mathbb{Z}p_{x}(f_{1})+\dots+ \mathbb{Z}p_{x}(f_{2n+1})\right)\subset\mathbb{Z}f_{1}+\dots+ \mathbb{Z}f_{2n+1}$.
  The same proof is used for the converse, by replacing $\Psi$ by $\Psi^{-1}$.
\end{proof}
\bigskip

\begin{prop}\label{p:012}$($\cite{mW}, Proposition 4.3$)$. Let $H = \mathbb{Z}x_{1}+\dots+\mathbb{Z}x_{p}$ with $x_{k } \in\mathbb{R}^{n}$.
 Then $H$ is dense in $\mathbb{R}^{n}$ if and only if for every
$(s_{1},\dots,s_{p})\in
 \mathbb{Z}^{p}\backslash\{0\}$ :
$$\mathrm{rank}\left[\begin{array}{cccc }
 x_{1} &\dots &\dots & x_{p} \\
   s_{1 } &\dots&\dots & s_{p }
 \end{array}\right] =\ n+1.$$
\end{prop}
\medskip

\begin{proof}[Proof of Theorem ~\ref{T:3}] Write $H_{x}=\mathbb{Z}f_{1}(x)+\dots+ \mathbb{Z}f_{2n+1}(x)$.
Since $\overline{H_{x}}=\mathbb{C}^{n}$, by Lemma~\ref{L:BBB017},
$(f_{1}(x),\dots, f_{2n}(x))$ is a basis of $\mathbb{C}^{n}$, so
$f_{1},\dots, f_{2n}$ are linearly independent over $\mathbb{R}$.
Denote by $E=vect(f_{1},\dots, f_{2n})$, then
$E=\Psi(\widetilde{L})$ and it has a dimension equal to $2n$ over
$\mathbb{R}$, so $\Psi:\widetilde{L}\longrightarrow E$ is an
isomorphism. Since $f_{2n+1}\in vect(f_{1},\dots, f_{2n})$ then by
Lemma~\ref{L:018},(i), $\overline{\mathbb{Z}p_{x}(f_{1})+\dots
+\mathbb{Z}p_{x}(f_{2n+1})}=\widetilde{L}$. Therefore:
\begin{align*}
 E& =\Psi(\widetilde{L})\\
 \ & =\Psi(\overline{\mathbb{Z}p_{x}(f_{1})+\dots+ \mathbb{Z}p_{x}(f_{2n+1})})\\
 \ & =\overline{\mathbb{Z}f(p_{x}(f_{1}))+\dots+ \mathbb{Z}f(p_{x}(f_{2n+1}))}\\
 \ & = \overline{\mathbb{Z}f_{1}+\dots+ \mathbb{Z}f_{2n+1}}\ \ \ \ (1)\\
 \end{align*}

 Let $1\leq k \leq 2n$ and $t_{k}\in \mathbb{R}^{*}$ such that $|t_{k}|<\frac{r_{\widetilde{G}}}{\|f_{k}\|}$.\
\\
$\bullet$ First, let's prove that $e^{t_{k}f_{k}}\in G$\
\\
Since $t_{k}f_{k}\in E$, then by (1), $t_{k}f_{k}\in
\overline{\mathbb{Z}f_{1}+\dots+ \mathbb{Z}f_{2n+1}}$. Thus there
exists a sequence $(g_{j})_{j\in\mathbb{N}}\subset
\mathbb{Z}f_{1}+\dots+ \mathbb{Z}f_{2n+1}$ such that
$\underset{j\to+\infty}{lim}g_{j}=t_{k}f_{k}$.
 By continuity of the exponential map we  have $\underset{j\to +\infty}{lim}e^{g_{j}}=e^{t_{k}f_{k}}$. Since
 $\mathbb{Z}f_{1}+\dots +\mathbb{Z}f_{2n+1}\subset exp^{-1}(\widetilde{G})$ then $g_{j}\in exp^{-1}(\widetilde{G})$, so
  $e^{t_{k}f_{k}}\in \widetilde{G}$, since $\widetilde{G}$ is closed in $Diff^{r}(\mathbb{C}^{n})$.

\
\\
$\bullet$ Second,  as $|t_{k}|<\frac{r_{G}}{\|f_{k}\|}$, then
$\|t_{k}f_{k}\|<r_{\widetilde{G}}$.\ Since $|t_{k}|\neq 0$ and
$(f_{1}(x),\dots, f_{2n}(x))$ is a basis of $\mathbb{C}^{n}$, so
is $(t_{1}f_{1}(x),\dots, t_{2n}f_{2n}(x))$. By the first step we
conclude that $e^{t_{k}B_{k}}\in G$ for every $k=1,\dots, 2n$. The
proof follows then from Theorem
  ~\ref{T:2}.
 \end{proof}
\medskip

The complex form of Proposition~\ref{p:012} is given in the
following:

\begin{prop}\label{p:012+}$($\cite{mW}, $page$ \ $35)$. Let $H = \mathbb{Z}z_{1}+\dots+\mathbb{Z}z_{p}$ with $z_{k } \in\mathbb{C}^{n}$
 and $z_{k} =
\mathrm{Re}(z_{k})+i \mathrm{Im}(z_{k})$, $k = 1,\dots, p$. Then
$H$ is dense in $\mathbb{C}^{n}$ if and only if for every
$(s_{1},\dots,s_{p})\in
 \mathbb{Z}^{p}\backslash\{0\}$ :
$$\mathrm{rank}\left[\begin{array}{cccc }
 \mathrm{Re}(z_{1}) &\dots &\dots & \mathrm{Re}(z_{p}) \\
  \mathrm{Im}(z_{1 }) &\dots &\dots & \mathrm{Im}(z_{p}) \\
  s_{1 } &\dots&\dots & s_{p }
 \end{array}\right] =\ 2n+1.$$
\end{prop}
\medskip

\begin{proof}[Proof of Corollary~\ref{C:5}]  The proof results directly, from Theorem~\ref{T:3} and Proposition ~\ref{p:012+}.
\end{proof}

\medskip

\begin{lem}\label{LLL:12}$($\cite{aAhM05}, Corollary 1.3$)$. Let $G$ be an abelian subgroup of
$GL(n,\mathbb{C})$.  If   $G$  has a locally dense orbit $\gamma$
in $\mathbb{C}^{n}$ then  $\gamma$ is dense in $\mathbb{C}^{n}$.
\end{lem}
\medskip

\begin{proof}[Proof of Corollary~\ref{C:6}] Since the matrices $B_{j}$, $1\leq j \leq 2n+1
$ commute then $\mathbb{Z}B_{1}+\dots+\mathbb{Z}B_{2n+1}\subset
exp^{-1}(G)$. Hence the proof of Corollary~\ref{C:6} results
directly from Corollary~\ref{C:5} and Lemma ~\ref{LLL:12}.
\end{proof}
\medskip
\
\\
{\bf Question1:} {\it How can we characterize explicitly
$\mathrm{g}=exp^{-1}(G)$ for any finitely generated abelian
subgroup $G$ of a lie group $\Gamma\subset
Diff^{r}(\mathbb{C}^{n})$?}
\\
\\
{\bf Question2:} {\it A somewhere dense orbit of a non abelian
subgroup of $Diff^{r}(\mathbb{C}^{n})$ can  always be dense in
$\mathbb{C}^{n}$?}

\medskip

\bibliographystyle{amsplain}

\providecommand{\bysame}{\leavevmode\hbox to3em{\hrulefill}\thinspace}
\providecommand{\MR}{\relax\ifhmode\unskip\space\fi MR }
\providecommand{\MRhref}[2]{%
  \href{http://www.ams.org/mathscinet-getitem?mr=#1}{#2}
}
\providecommand{\href}[2]{#2}
\begin{thebibliography}{}

\end{thebibliography}


\begin{thebibliography}{9}



\bibitem{aAhM05} A.Ayadi and H.Marzougui, \emph{Dynamic of abelian subgroups of GL(n, C):
a structure Theorem}, Geometria Dedicata, 116(2005)111-127.

\bibitem{aAh-M05}  A.Ayadi and H.Marzougui, \emph{Dense orbits for abelian subgroups of
GL(n, C)}, Foliations 2005: 47-69. World Scientific,Hackensack,NJ,
2006.

\bibitem{IL} A.Kirillov, \emph{Introduction to Lie Groups and Lie Algebras}, Department of Mathematics, SUNY at Stony Brook, Stony Brook, NY 11794,
USA.

\bibitem{ASRW} A.Sagle and R. Walde, \emph{Introduction to Lie groups and Lie
algebras}, volume 51, (1973),(Academic Press, '73).

\bibitem{mW} M.Waldschmidt, Topologie des points rationnels.,
Cours de troisi\`eme Cycle, Universit\'e P. et M. Curie (Paris
VI), 1994/95.


\bibitem{WDM} W.De Melo, \emph{Differential Topology notes},
course, IMPA - Instituto de Matemática Pura e Aplicada, 2012.


\bibitem{YNAA-C} Y.N'dao and A.Ayadi, \emph{Chaoticity and regular action of diffeomorphisms group of $K^n$},
 preprint ArXiv, 1208.6395-(2012).

\bibitem{YNAA-D} Y.N'dao and A.Ayadi, \emph{The dynamic of abelian subgroup of $diff^r(K^n)$, fixing a point (K=R or C)},
 preprint ArXiv, 1207.6466-(2012).
\end{thebibliography}
\vskip 0,4 cm

\end{document}